\documentclass[11pt, twoside]{article}
\usepackage{amsmath,amsthm,amssymb}
\usepackage{times}
\usepackage{enumerate}
\usepackage{latexsym}
\usepackage{amssymb}
\usepackage{amsmath}
\usepackage{fancybox}
\usepackage{epstopdf}
\usepackage{wasysym}
\usepackage{epsfig}
\usepackage{float}
\usepackage{amsfonts}
\usepackage{srcltx}
\usepackage{esint}
\usepackage[dvipsnames]{xcolor}
\usepackage{enumerate}

\numberwithin{equation}{section}

\def\titlerunning#1{\gdef\titrun{#1}}

\makeatletter

\def\author#1{\gdef\autrun{\def\and{\unskip, }#1}\gdef\@author{#1}}

\def\address#1{{\def\and{\\\hspace*{18pt}}\renewcommand{\thefootnote}{}%
\footnote {#1}}%
\markboth{\autrun}{\titrun}}

\makeatother

\def\email#1{e-mail: #1}

\def\keywords#1{\par\medskip
\noindent\textbf{Keywords.} #1}

\def\subjclass#1{\par\medskip
\noindent\textbf{Mathematics Subject Classification (2010).} #1}

\theoremstyle{plain}
\newtheorem{theorem}{Theorem}[section]
\newtheorem{proposition}[theorem]{Proposition}
\newtheorem{lemma}[theorem]{Lemma}

\theoremstyle{definition}
\newtheorem{defin}[theorem]{Definition}

\newtheorem{remark}[theorem]{Remark}
\newtheorem{example}{Example}

\theoremstyle{remark}

\def\be{\begin{equation}}
\def\ee{\end{equation}}

\def\bk{\color{black}}

\def\huz{H^1_0(\Omega)}

\def\into{\int_{\Omega}}
\def\dys{\displaystyle}
\def\vare{\varepsilon}

\def\luo{L^{1}(\Omega)}
\def\rn{\mathbb{R}^{N}}
\def\re{\mathbb{R}}

\begin{document}

\titlerunning{Finite and Infinite energy solutions of singular elliptic problems}
\title{Finite and Infinite energy solutions of singular elliptic problems: Existence and Uniqueness}

\author{Francescantonio Oliva \and Francesco Petitta}	

\date{}

\maketitle
%
%
%
	
	\address{F. Oliva, F. Petitta: SBAI Department,
 Sapienza University of Rome, Via Scarpa 16, 00161 Roma, Italy; \email{francesco.oliva@sbai.uniroma1.it}, \email{francesco.petitta@sbai.uniroma1.it}
}

		
	\begin{abstract}
We establish existence and uniqueness of solution for  the homogeneous  Dirichlet problem associated to a fairly general class of   elliptic equations modeled by
$$
 -\Delta u= h(u){f} \   \ \text{in}\,\  \Omega, 
$$
where $f$ is an irregular datum, possibly a measure,  and $h$ is a continuous function that may blow up at zero. We also provide  regularity results on both the solution and the lower order term depending on the regularity of the data, and we discuss their optimality.


\end{abstract}

\section{Introduction}
Let $\Omega$ be a bounded and  smooth open subset of $\rn$, and  consider, as a model, the following singular elliptic boundary-value problem
\begin{equation}
\begin{cases}
\displaystyle -{\rm div} (A(x) \nabla u)= \frac{f}{u^{\gamma}} &  \text{in}\, \Omega, \\
u=0 & \text{on}\ \partial \Omega,
\end{cases}\label{1in}
\end{equation}
where $\gamma>0$, $A(x)$ is a bounded  elliptic matrix, and $f$ is nonnegative. 

\medskip
  Physical motivations in the study of problems as  \eqref{1in} arise, for instance,  in  the study of thermo-conductivity where ${u^{\gamma}}$ represents the resistivity of the material (\cite{fm}),   in signal transmissions (\cite{no}), and in   the theory  non-Newtonian pseudoplastic fluids (\cite{nc}).  See also \cite{var} (and references therein) for a precise  description of a model of boundary layers in which equations as in \eqref{1in}  also appear.

From the purely theoretical  point of view, after the first pioneering existence and uniqueness result given in  \cite{fm}, a systematic treatment of problems as   \eqref{1in} was developed starting from \cite{stu,crt}.

Consider, for simplicity,  $A(x)=I$, i.e. the case of the laplacian as principal operator; if  $f$ is smooth enough  (say H\"older continuous)  and bounded away from zero on $\Omega$ then the existence and uniqueness of a classical solution  to \eqref{1in} is proven by desingularizing the problem and then by applying a  suitable sub- and super-solution method.  Some remarkable refinements  of the previous results were given in \cite{lm}; here, the authors proved, in particular,  that $u\not \in C^{1}(\overline{\Omega})$ if $\gamma>1$ and it has finite energy, i.e. $u\in \huz$, if and only if $\gamma<3$ (see also \cite{gl} for further insights). 

Classical theory for equations as in  \eqref{1in}, also called singular Lane--Emden--Fowler equations, has been also extended  to  the case in which the term ${s^{-\gamma}}$ is replaced by a $C^{1}$ nonincreasing nonlinearity $h(s)$ that  blows up at zero  at a given rate (see \cite{stu,crt,ls,zc}). 

\medskip

More general situations can be considered. Let $f$ be a  nonnegative  function belonging to some $L^m(\Omega)$, $m\geq1$, or even,  possibly,  a measure.  If $f\in \luo$, in \cite{bo}, the existence of a distributional solution $u$ to \eqref{1in}  is proved. In particular the authors  prove that a locally strictly positive  function $u$  exists such that  equation in \eqref{1in} is satisfied in the sense of distributions: moreover $u\in W^{1,1}_0(\Omega)$ if $\gamma<1$, $u\in\huz$ if $\gamma=1$, and $u\in H^1_{loc} (\Omega)$  if $\gamma>1$, where,  in the latter case,  the boundary datum is only assumed in a weaker sense than the usual one of traces, i.e. $u^{\frac{\gamma+1}{2}}\in H^1_{0} (\Omega)$. Note that, if $\gamma>1$, then  solutions with infinite energy do exist, even for smooth data (\cite{lm}).  Let us also   mention \cite{gmp,gmp2} where,  in order to deal with homogenization issues, existence and uniqueness of finite energy solutions  are considered  for $f\in L^{m}(\Omega)$, $m>\frac{N}{2}$,   also in the case of a continuous  nonlinearity $h(s)$ that mimics ${s^{-\gamma}}$.

In the  case of $f$ being a measure  the situation becomes striking different. Nonexistence of solutions to problem \eqref{1in} is proven (at least in the sense of approximating sequences) in \cite{bo} if the measure is too concentrated, while  in  \cite{po} sharp existence results are obtained in the measure is diffuse enough; here concentration and diffusion is intended in the sense of capacity. For general, possibly singular, measures data existence of a distributional/renormalized solution is considered in \cite{do} also in the case of a more general (not necessarily monotone) nonlinearity.

Without the aim to be exhaustive we also refer the reader to the following papers and to the references cited therein in which various relevant extensions and refinements (nonhomogeneous case, variational approach, nonlinear principal operators, natural growth terms, etc, ...) are considered: \cite{t,gow,cd, c,a6,gps1,gps2,cgs,cc,sz,edr, dc,cst}.

\medskip 

{Consider now  the boundary-value problem 
\begin{equation}
\begin{cases}
\displaystyle Lu= h(u)f &  \text{in}\, \Omega, \\
u=0 & \text{on}\ \partial \Omega,
\end{cases}\label{luno}
\end{equation}
where $L$ is a linear elliptic operator in divergence form,  $h: \re^+ \to \re^+$ is a continuous function that  may \bk blow  up at $s=0$ and possesses a limit at infinity, and  $f$ is a nonnegative  function in $L^{1}(\Omega)$ (or, possibly, a bounded Radon  measure on $\Omega$).

Besides the one  arising from the  presence of  possibly \bk a measure datum, new  difficulties have to be taken into account in this general framework; even  if $f$ is only a nonnegative function, then the solutions do not belong in general to $\huz$ even for small $\gamma$ (see Example \ref{luigi} below) nor the lower order term in \eqref{luno} need to belong to $L^1(\Omega)$ (Example \ref{ex}).

  In particular, as also observed in \cite{sz}, the question of the  summability properties of the lower order term in \eqref{luno} plays a crucial role in order to deal with uniqueness of  solutions (see \cite{po,op}, for some partial results in this direction).  In general, in fact, only finite energy solutions are known to be unique, at least in the model case \eqref{1in} (see \cite{boca,cst}). 

\medskip

In the present paper, under fairly general assumptions we introduce a natural  notion of distributional  solution to problem \eqref{luno} for which  existence can be shown to hold. Moreover,  uniqueness holds  provided $h$ is nonincreasing. As we already mentioned, particular care will  be addressed on how the homogeneous boundary datum for the solution $u$ is (weakly) attained. 

 If $f$ is a function in $L^{m}(\Omega)$, $m\geq 1$, we also investigate the question of whether the solution  to problem \eqref{luno} has finite energy. We provide several instances of this occurrence  depending on the regularity of the datum and on the behavior of $h(s)$ both at zero and at infinity.   Additionally, we obtain sharp thresholds for the lower order term to belong to $L^{1}(\Omega)$. The results and their optimality are  discussed  through appropriate  examples.  

Also considering the obstructions given by the above mentioned examples,   we will finally be lead to establish a  weighted  summability estimate on the lower order term  $h(u)f$;  this will be a key tool in order to prove  uniqueness, which will be obtained by mean of a suitable Kato type inequality.

\medskip 

{The plan of the paper is as follows:
in Section \ref{sec2} we precise the structural assumptions we are going to  work with, we give our definition of solution for problem \eqref{luno}, and we state our existence (Theorem \ref{esiste}) and uniqueness (Theorem \ref{uniquenessmain}) results; we then provide some useful preliminary tools we will need. In particular, Section \ref{veron} contains some  outcomes \bk concerning the linear case and a Kato type inequality with measures. Section \ref{finite} consists of an   extenstive \bk account on the case of finite energy solutions; moreover,  in Section \ref{furt} we also present a prototypical example and we address the question of the integrability of the lower order term $h(u)f$.  Section \ref{uniqueness_distributional} is then devoted to the proof of Theorems \ref{esiste} and \ref{uniquenessmain}. Finally, in Section \ref{furte} we discuss some further examples and possible extensions. 
 }

\subsection{Notations and auxiliary functions}

We will use the following well known auxiliary functions defined for fixed $k>0$  
$$T_k(s)=\max (-k,\min (s,k)), \ \ \ \ \ G_k(s)=(|s|-k)^+ \operatorname{sign}(s),$$
with  $s\in\re$. Observe that $T_{k}(s)+G_{k}(s)=s$, for any $s\in \re$ and $k>0$.  

 We denote the distance of a point $x\in \Omega$ from the boundary of the set in the following way 
$$\delta(x):= \operatorname{dist}(x, \partial \Omega)\,.$$
As $\partial \Omega$ is smooth enough we will use systematically the fact that 
$$
\int_{\Omega} \delta(x)^r\ dx<\infty\,,
$$
if and only if $r>-1$.  Moreover, with a little abuse of notation, in order to avoid technicalities and without loosing  generality, we can refer to  $\delta(x)$ as a suitable positive smooth (say $C^{1}$) modification of the distance function which agrees with $\delta (x)$ in a neighbourhood of $\partial \Omega$.

We also  need to define the following $\vare$-neighborhood of $\partial\Omega$: 
$$\Omega_{\varepsilon}=\{x\in\Omega : \delta(x)<\varepsilon\},$$ that we will always assume to be smooth (up to the choice of a suitable small $\vare$).

The space of bounded Radon measures  will be denoted by $\mathcal{M}(\Omega)$, while we will also  made  use of the following weighted spaces 
$$
\dys L^1(\Omega,\delta)=\left\{f\in L^{1}_{loc}(\Omega): \into |f| \delta\ dx <\infty\right\}\ \ \text{and}\ \ \mathcal{M}(\Omega,\delta)=\left\{\mu\in  \mathcal{M}_{loc}(\Omega): \into  \delta\ d|\mu| <\infty\right\}\,.
$$ 

For an integer $j$,  
$$C^j_0(\overline{\Omega}):= \{\phi\in C^j(\overline{\Omega}): \phi=0 \text{  on  }\partial \Omega\},$$
will denote the space of $C^{j}$ functions that vanishes at $\partial\Omega$. As usual, subscript $c$ will indicate a space of function with compact support in $\Omega$, e.g. $C_c (\Omega)$.

If no otherwise specified, we will denote by $C$ several constants whose value may change from line to line and, sometimes, on the same line. These values will only depend on the data but they will never depend on the indexes of the sequences we will often introduce. 
Moreover for the sake of simplicity we use the simplified notation $$\int_{\Omega} f:=\int_{\Omega} f(x)\ dx\,,$$ when referring to integrals when no ambiguity is possible.

\section{Setting of the problem  and some preliminary tools}\label{sec2}
\setcounter{equation}{0}

\subsection{Main assumptions and setting of the problem} Let $\Omega$ be a open and bounded domain  of $\rn$, $N\geq 2$, with smooth boundary, and consider the following homogeneous boundary-value problem
\begin{equation}
\begin{cases}
\displaystyle Lu= h(u)\mu &  \text{in}\, \Omega, \\
u=0 & \text{on}\ \partial \Omega,
\end{cases}\label{1}
\end{equation}
where $L$ is a strictly elliptic linear operator in divergence form, that is
$$
Lu:=-\operatorname{div} (A(x)\nabla u)\,,
$$
where $A$ is a matrix with Lipschitz continuous coefficients such that 
\begin{equation}
A(x)\xi \cdot \xi \ge \alpha|\xi|^2\,\  \ \text{and}\ \  \ |A(x)|\le \beta\,, \label{lipA}\,
\end{equation} 
\noindent for every $\xi$ in $\mathbb{R}^N$, for almost every $x$ in $\Omega$ and for $\alpha,\beta>0$.   As we will see,  the  smoothness we assume on both $\partial\Omega$ and the coefficients of the matrix $A(x)$ are necessary in order to use standard $L^{p}$ elliptic regularity theory,  Hopf's  boundary point lemma, and some further technical devices;  we shall stress as some of the results we present require less regularity assumptions. \bk
\medskip

On the nonlinearity $h:\re^{+}\to\re^{+}$ we assume that it is continuous,  such that \begin{equation}\displaystyle \lim_{s\to 0^+}  h(s)\in (0,\infty]\bk\,,\ \ \ \text{and} \ \ \lim_{s\to \infty} h(s):=h(\infty)<\infty\,.\label{infty}\end{equation}  We also assume the following growth condition near zero
\begin{equation}\label{h1} 
\displaystyle \exists\;\ {K_1},\underline{\omega}>0\;\ \text{such that}\;\  h(s)\le \frac{K_1}{s^\gamma} \ \ \text{if} \ \ s<\underline{\omega},
\end{equation}
with $\gamma>0$.  We explicitly remark that we do not assume any control from below on the function  $h$ so that the case of a bounded and continuous function satisfying \eqref{infty} is allowed. 

Finally, $\mu$ will be a nonnegative measure in $\mathcal{M}(\Omega)$.  In particular, recall that $\mu$ can be (uniquely) decomposed as
$$\mu=\mu_d + \mu_c,$$
where, resp.  $\mu_d\ge 0$ is absolutely continuous with respect to the harmonic capacity (also called $2$-capacity) and $\mu_c\ge 0$ is concentrated on a set of zero $2$-capacity  (see for instance \cite{fuku}). We refer the reader to \cite{hkm}  for an exhaustive introduction to the theory of capacity (see also \cite{dmop} which contains a nice introduction to the topic which is sufficient to our purposes). Also, we will always assume, without loss of generality, that $\mu_{d}\not\equiv 0$. Otherwise, in fact, the singularity of the problem virtually disappears and one is  
brought to nonsingular problem  (see Remark \ref{weakdef} below,  and  the discussion in \cite[Section 5]{do}). 

If $\mu_{c}=0$ we will call $\mu$ a diffuse measure in $\mathcal{M}(\Omega)$. Recall (see \cite{bgo}) that, if $\mu$ is diffuse, then the following decomposition holds in the sense of distributions 
\begin{equation}\label{bgodec}
\mu=g-{\rm div\,}G\,,
\end{equation}
where $g\in L^1(\Omega)$ and $G\in L^2(\Omega)^N$. It is proven in \cite{dmop} that such a $\mu$ can be approximated in the narrow topology of measures by  smooth functions $\nu_{n}$, that is 
\begin{equation}\displaystyle \lim_{n\to\infty}\label{narrowvera}\int_{\Omega} \nu_n \phi = \int_{\Omega} \phi d\mu, \ \ \ \forall \phi \in C(\overline{\Omega})\,,\end{equation}
such that 
$$
\nu_{n}=g_{n}-{\rm div\,}G_{n}\,,
$$
where $g_n$ weakly converges in $L^1(\Omega)$ to $g$ and $G_n$ strongly converges in $L^2(\Omega)^N$ to $G$. 

Concerning the boundary datum, our solutions are not expected in general to belong to $W^{1,1}_{0}(\Omega)$ (see for instance \cite{bo,cst, po}). Due to this fact, we shall need to infer that our  (positive) solution  $u\in L^{1}(\Omega)$ assumes the value zero at $\partial \Omega$ in the following sense 
\begin{equation}\label{bd}
\displaystyle \lim_{\vare \to 0}\frac{1}{\vare}\int_{\Omega_{\vare}} u =0\,,
\end{equation}
which is known to be weaker than  the classical sense of traces for functions in  $W^{1,1}_{0}(\Omega)$ (see for instance \cite{sel,afp}).

\medskip 

To deal with both  existence and uniqueness for solutions to our problem we give the following variant of the definition of distributional solution to \eqref{1}  given in \cite{do} (and inspired by \cite{mp}). 
\begin{defin}\label{def}
A positive function  $u\in L^{1}(\Omega)\cap W^{1,1}_{loc}(\Omega)$ is a \emph{distributional solution} to problem \eqref{1} if $  h(u) \in L^1_{loc}(\Omega,\mu_d)$,  \eqref{bd} holds, and  following is satisfied  
	\begin{equation} \displaystyle \int_{\Omega}A(x)\nabla u \cdot \nabla \varphi =\int_{\Omega} h(u)\varphi\mu_d + h(\infty)\int_{\Omega} \varphi\mu_c \ \ \ \forall \varphi \in C^1_c(\Omega).\label{weakdef3}\end{equation}
\end{defin}  
\begin{remark}\label{weakdef}
Asking for  $ h(u) \in L^1_{loc}(\Omega,\mu_d)$  gives sense to  the right hand side of \eqref{weakdef3}, while, as $u$ is in $ W^{1,1}_{loc}(\Omega)$,  the left hand side makes sense as well. 

Recall we are assuming $\mu_{d}\not\equiv 0$; if this is not the case, in fact, \eqref{weakdef3} becomes linear and it can be treated with classical tools (see Section \ref{veron}). 

If $h(\infty)=0$ then the concentrated part of the measure disappears  in \eqref{weakdef3}, so we are just solving problem \eqref{1} with datum $\mu_{d}$; this reflects a well known concentration type phenomenon of the approximating sequences of solutions when the datum is too concentrated. With this sort of nonexistence result in mind (see \cite{bo, po, do} for further details on these phenomena),   we will always understand that $h(\infty)\not = 0$ if $\mu_{c}\not\equiv 0$ in order to be consistent with the fact that we are solving \eqref{1}. \bk 
\end{remark}

Our main existence and uniqueness results are the content of   the following theorems whose  proofs will be given in Section \ref{uniqueness_distributional}.
\begin{theorem}\label{esiste}
Let $h$ be a continuous function satisfying \eqref{infty} and \eqref{h1},  and let $\mu\in \mathcal{M}(\Omega)$  be a nonnegative measure. Then a solution to problem \eqref{1} in the sense of Definition \ref{def} does exist. 
\end{theorem}
\begin{theorem}\label{uniquenessmain}
Under the same  assumptions of Theorem \ref{esiste}, if $h$ is nonincreasing,  then the distributional  solution to problem \eqref{1} is unique.
\end{theorem}

\medskip

\subsection{The linear case and a Kato type lemma}
\label{veron}
\noindent Here we briefly discuss the linear case  (say $h\equiv 1$) of problem \eqref{1} with data in a suitable weighted class.  We first recall the existence and uniqueness of a solution to this problem obtained by a transposition argument (see \cite{v}, see also \cite{mv,dr3}). 
Let $A$ be a matrix with Lipschitz continuous coefficients satisfying assumption \eqref{lipA} and consider the problem   
\begin{equation}
\begin{cases}
\displaystyle -\operatorname{div} (A(x)\nabla u)= \mu &  \text{in}\, \Omega, \\
u=0 & \text{on}\ \partial \Omega,
\end{cases}\label{pbVeron}
\end{equation}
where $\mu$ belongs to $\mathcal{M}(\Omega,\delta)$.  With the symbol $L^{*}$ we will indicate the transposed operator defined by
$$
L^{*}\varphi :=- {\rm div} (A^{*}(x)\nabla \varphi)\ \ \text{in}\ \ \mathcal{D}'(\Omega)\,.
$$
\begin{defin}\label{verondefin}
	A \emph{very weak solution} of problem \eqref{pbVeron} is a function $u\in L^1(\Omega)$ such that the following holds
	
	\begin{equation*} \displaystyle \int_{\Omega}u L^{*}\varphi  = \int_{\Omega} \varphi d\mu,\label{veryweakdef1}\end{equation*}
	
	\noindent for every  $\varphi \in C^1_0(\overline{\Omega})$ such that $L^{*}\varphi\in L^\infty(\Omega)$. 
	\label{veryweakdef}
\end{defin}

We have the following two  results whose proofs can be found respectively in \cite[Corollary 2.8]{v} and   \cite[Theorem 2.9]{v}. 

\begin{lemma}\label{veroncomplocale}
	Let $f\in L^1_{loc}(\Omega)$ and let $u\in L^{1}_{loc}(\Omega)$ such that 
	\begin{equation*}
		\displaystyle  \int_{\Omega}u L^{*} \varphi = \int_{\Omega} f\varphi,
	\end{equation*}
	\noindent for every  $\varphi \in C^1_c ({\Omega}\bk)$ such that $L^{*} \varphi\in L^\infty(\Omega)$. Then for every open subsets $G \subset\subset  G' \subset  \subset \Omega$,  and for every $1\leq q< \frac{N}{N-1}$,  we have 
	\begin{equation*}
		||u||_{W^{1,q}(G)} \le C \left( ||f||_{L^1(G')}+||u||_{L^1(G')}\right)\,.
	\end{equation*}	\bk
\end{lemma} 

\begin{theorem}\label{exi}
	Let $\mu\in\mathcal{M}(\Omega,\delta)$ then there exists a unique very weak solution to problem \eqref{pbVeron}.
\end{theorem}

\medskip

In order to prove uniqueness for the singular problem \eqref{1} one also needs a  Kato type inequality for variable coefficients operators  and diffuse measures as data, which is inspired by  \cite{bp} (see also \cite{kl} for a related result in the context of Dirichlet forms).  First we need the following:

\begin{lemma}\label{TroncataH1loc}
	Let $\mu\in\mathcal{M}(\Omega,\delta)$ and let $u$ be the  very weak solution to problem \eqref{pbVeron}, then 
	$T_k(u)\in H^1_{loc}(\Omega)$. Moreover,  for any $E\subset\subset \Omega$, one has 
	$$
	\|T_k(u)\|_{H^{1}(E)}\leq C, 
	$$
 where $C$ is a constant that only depends on $\alpha, E, \Omega$ and $k$. 
\end{lemma}
\begin{proof}
	We consider the scheme of approximation used in \cite{v}, that is 
	\begin{equation}
	\begin{cases}
	\displaystyle L u_n= f_n &  \text{in}\, \Omega, \\
	u_n=0 & \text{on}\ \partial \Omega,
	\end{cases}\label{pbnVeron}
	\end{equation}
	where $f_n$ is a sequence of  smooth functions bounded in  $L^1(\Omega,\delta)$ \bk converging to $\mu$ in the following sense 
	\begin{equation}\displaystyle \label{narrow} \lim_{n\to\infty}\int_{\Omega} f_n \phi = \int_{\Omega} \phi d\mu, \ \ \ \forall \phi: \frac{\phi}{\delta} \in C(\overline{\Omega}).\end{equation}
	 In \cite[Theorem 2.9]{v} the author shows that the classical solutions  $u_n$  satisfy
	\begin{equation}\label{lu}
	\|u_{n}\|_{\luo}\leq C\,,
	\end{equation}
	and  that $u_{n}$ converges a.e. towards the solution to problem \eqref{pbVeron}. \bk 
	Let $E\subset\subset\Omega$ and consider a function $\varphi\in C^\infty_c(\Omega)$ such that $0\le\varphi\le 1$ and $\varphi = 1$ on  $E $.  Testing   \eqref{pbnVeron} with $T_{k}(u_{n})\varphi$ we have 
\begin{equation}\label{AppendiceTroncateinH1loc1}\begin{array}{l}
	\dys\alpha\int_{E} |\nabla T_k(u_n)|^2\leq 	\alpha\int_{\Omega} |\nabla T_k(u_n)|^2\varphi \le \int_{\Omega}  A(x)\nabla u_n \cdot \nabla T_k(u_n) \varphi \\ \\ \dys 
		=   \int_{\Omega} T_k(u_n) f_{n} \varphi - \int_{\Omega} T_k(u_n) A(x)\nabla u_n \cdot  \nabla \varphi\,. 
		\end{array}
	\end{equation}
	On one hand, the first term on the right hand side  satisfies
	$$
	\left| \int_{\Omega} T_k(u_n) f_{n} \varphi \right| \leq  k\int_{\{\rm supp\varphi \}}f_{n} \varphi\leq Ck\,. 
	$$
	On the other hand, the second term on the right hand side of \eqref{AppendiceTroncateinH1loc1} gives 
$$\begin{array}{l}
		\dys\int_{\Omega} T_k(u_n) A(x)\nabla u_n\cdot \nabla \varphi =  \int_{\Omega} T_k(u_n) u_n L^{*} \varphi- \int_{\Omega} T_k(u_n) \nabla T_k(u_n)\cdot A^*(x)\nabla \varphi   \\\\ \nonumber
	\dys	= \int_{\Omega} T_k(u_n) u_nL^{*} \varphi - \frac{1}{2}\int_{\Omega}\nabla [T_k(u_n)]^2\cdot   A^*(x)\nabla \varphi  \\ \\ \nonumber
		\dys =  \int_{\Omega} T_k(u_n) u_n L^{*} \varphi - \frac{1}{2}\int_{\Omega}  [T_k(u_n)]^2L^{*} \varphi   \\ \\  \nonumber
	\dys	=  \int_{\Omega} T_k(u_n) (u_n- \frac{1}{2}T_k(u_n))L^{*} \varphi \leq  k \int_{\Omega} |u_n| |L^{*} \varphi|\leq Ck \int_{\Omega} |u_n|.
		\end{array}
$$
	Gathering the previous estimate with \eqref{AppendiceTroncateinH1loc1}, recalling \eqref{lu} and also using the weak lower semicontinuity, we finally obtain 
$$
	\int_{E}|\nabla T_k(u)|^2\leq 	\liminf_{n\to\infty}\int_{E}|\nabla T_k(u_n)|^2 \le   Ck\,.
$$ 
\end{proof}
Note, in particular,  that, if $\mu$ is diffuse, then, by standard capacity properties,   $T_{k}(u)$  (and so $u$) is defined  $|\mu|$-a.e. 
The following is a version of Kato inequality in the case of a diffuse measure as datum. 
\begin{lemma}\label{Katoinequalitylocal}
Let $u$ be the  very weak solution to problem \eqref{pbVeron} where $\mu$ is a  diffuse measure in  ${\mathcal M}(\Omega,\delta)$\bk. Then
\begin{equation*}
	\displaystyle  \int_{\Omega} u^+ L^{*}\varphi \le \int_{\{u\ge 0\}} \varphi \mu,
\end{equation*}	
for any nonnegative $\varphi\in C^1_c(\Omega)$ such that $L^{*}\varphi \in L^\infty(\Omega)$.
\end{lemma}
\begin{proof}
 We  consider again  the approximating scheme defined by   \eqref{pbnVeron} but, to our purposes, we need to specify the structure of the approximating sequence $f_{n}$. As $\delta\mu$ is a diffuse measure in ${\mathcal M}(\Omega)$, by \eqref{bgodec}, we have that $\delta\mu=g-{\rm div \,}G$ where $g\in L^{1}(\Omega)$ and $G\in L^{2}(\Omega)^{N}$,  and it  can be approximated in the narrow topology of measures by $g_{n}-{\rm div\,}G_{n}$,  
where $g_n$ weakly converges in $L^1(\Omega)$ to $g$ and $G_n$ strongly converges in $L^2(\Omega)^N$ to $G$. Recalling  \eqref{narrowvera}, it is not difficult to establish that, if we define  $f_{n}$ through 
$$
\delta f_{n}=g_{n}-{\rm div\,}G_{n}
$$
then $f_{n}$ is an  approximation of  $\mu$ in the sense of \eqref{narrow}.
\bk

Now, let  $\Phi :\mathbb{R} \to \mathbb{R}$ be a $C^2$-convex function with $0\le \Phi' \le 1$ and $\Phi''$ with compact support (also $\Phi(0)=0$)\bk;  letting $\varphi\in C^1_c(\Omega)$ be a nonnegative function such that $L^{*}\varphi \in L^\infty(\Omega)$, we have 
$$\begin{array}{l}
	\displaystyle  \int_{\Omega} \Phi(u_n)L^{*}\varphi =  \int_{\Omega} \nabla u_n \cdot A^*(x)\nabla \varphi \Phi'(u_n) \leq  \int_{\Omega} A(x) \nabla u_n \cdot \nabla \left( \varphi \Phi'(u_n)  \right) \\ \\ \dys - \alpha \int_{\Omega} \Phi''(u_{n})|\nabla u_{n}|^{2} \varphi 
	  \leq \bk  \int_{\Omega} \varphi \Phi'(u_n) f_n\,,
\end{array}$$
where we  used both \eqref{lipA} and  the convexity of $\Phi$.

 Now we  pass to the limit with respect to  $n$ in 
 \begin{equation}
 \label{katolocale1}
 	\displaystyle  \int_{\Omega} \Phi(u_n)L^{*}\varphi  \leq \bk  \int_{\Omega} \varphi \Phi'(u_n) f_n\,.
 \end{equation}
 Thanks to Lemma \ref{veroncomplocale} we have local strong convergence in $L^{1}(\Omega)$ (at least) and then, using also that $\Phi''$ has compact support,  we pass to the limit by dominated convergence theorem on the  left hand side. 
Now observe that $\Phi'(u_n)$  converges to $\Phi'(u) $  in $L^{\infty} (\Omega)$ $\ast$-weak and a.e. so that, due to the structure of $\delta f_{n}$, in order to pass to the limit  in the  right hand side of \eqref{katolocale1} it suffices to check that  $\Phi'(u_n)$ is locally bounded in $H^{1}_{loc}(\Omega)$ (and then using weak compactness). To do that only observe that, as  $\Phi''$ has compact support, one has 
$$\nabla \Phi'(u_n)= \Phi''(u_n) \nabla u_n = \Phi''(u_n)\nabla T_k(u_n),$$
for some level $k$, then we use Lemma  \ref{TroncataH1loc} to conclude.

 We can then deduce that 
\begin{equation}\label{katolocale2}
\displaystyle  \int_{\Omega} \Phi(u)L^{*}\varphi \leq  \int_{\Omega} \varphi \Phi'(u) d\mu\,.
\end{equation}
Hence, we conclude  by taking  a sequence of regular convex functions $\Phi_\vare (t)$ such that  $\Phi_\vare (t) =t$ if $t\geq 0$ and $|\Phi_{\vare}(t)|\leq \vare$ if $t<0$. One can pass to the limit in \eqref{katolocale2} (with $\Phi_{\vare}$ instead of $\Phi$)  finally getting
\begin{gather*}
\displaystyle  \int_{\Omega} u^+L^{*}\varphi \leq  \int_{\{u\ge 0\}} \varphi d\mu\,.
\end{gather*}

\end{proof}

\subsection{Further useful tools}

In Section \ref{finite} we will deal with self-adjoint operators $L$ and we shall use the properties of the first eigenfunction related to $L$ defined as the function  $\varphi_1\in H^{1}_{0}(\Omega)\cap W^{2,p}(\Omega)$ ($1\leq p<\infty$)
such that
\begin{equation*}
\begin{cases}
\displaystyle L \varphi_1= \lambda_1 \varphi_1 &  \text{in}\, \Omega, \\
\varphi_1=0 & \text{on}\ \partial \Omega.
\end{cases}\label{evpb}
\end{equation*}
Moreover, it is possible to prove the following consequence of  Hopf's boundary point lemma (see for instance \cite[Lemma 2]{dr});  there exist two positive constants $c_{1 }$ and $c_{2}$ such that
\begin{equation}
c_{1}\delta(x) \le \varphi_1(x) \le c_{2}\delta(x), \ \ \ \text{for } x \in \Omega.
\label{hopfvarphi2}
\end{equation}

Given a continuous function $h$ satisfying assumptions \eqref{infty} and \eqref{h1},  in order  to use some comparison arguments, we will need to construct  two  nonincreasing auxiliary continuous functions $\underline{h}, \overline{h}:\re^{+}\to\re^{+}$ such that 
\begin{equation}\label{uoh}
\underline{h}(s)\leq h(s)\leq  \overline{h}(s)\,\ \ \ \text{for any}\ s>0\,.
\end{equation}
 The construction of $\underline{h}$ is given in \cite{do} in such a way that  it also satisfies $\underline h(s)\leq T_{n}(h(s))$, for any positive $s$ and any $n\in \mathbb{N}$.  The construction of $\overline{h}$ is also easy; for instance one can  pick  $\rho\leq \underline\omega$ such that  
			$$\frac{K_1}{\rho^\gamma} \ge \sup_{s\in [\underline\omega,\infty)}h(s)\,,$$
one can let  
			$$i_0=\frac{K_1}{\rho^\gamma},\ \ \ \ \ \dys i_m= \sup_{s\in [\rho + m - 1,\infty)}h(s), \ \ m\ge 1\,,$$
			and define
			\begin{equation}\label{oh}
			\begin{array}{l}
			\dys \overline{h}(s):=  \frac{K_1}{s^\gamma}\chi_{\{(0,\rho)\}}(s) + \sum_{m\geq1}\Big(2(i_{m}-i_{m-1})\left(s-\rho-m+1\right)+i_{m-1}\Big)\chi_{\{[\rho + m-1,\rho + m- \frac{1}{2})\}}(s)  \\ \\ \dys + i_m\chi_{\{[\rho + m -\frac{1}{2},\rho +m]\}}(s).
			\end{array}\end{equation}			
Observe that, by construction, also $\overline{h}$ satisfies \eqref{h1} with constants $\gamma$,  $K_{1}$, and $\rho$ instead of $\underline\omega$.  
\bk

\section{Finite energy solutions}

\label{finite}
\setcounter{equation}{0}

  Here we analyze the issue of whether    \eqref{1}  admits a solution in $\huz$. The results will depend on   both the regularity of the data and the behavior of the nonlinearity $h$. Although, as we will stress below,  some instances of finite energy solutions can also be considered in the  case of measure data and more general operators,  in this section,  for simplicity,  we restrict our attention  to the case of a nonnegative datum in some $L^{m}(\Omega)$, with $m\geq1$,  and a self-adjoint operator $L$ (i.e. we assume $A(x)$ is symmetric).   Moreover, if $h(\infty)\not= 0$ only truncations belonging  in the energy space are expected (\cite{do}), so that we will  also assume throughout this section  that \eqref{infty} is satisfied with $h(\infty)= 0$. 

Therefore,  we look for solutions in $\huz$ for the following problem
\begin{equation}
\begin{cases}
\displaystyle -{\rm div}(A(x) \nabla u)= h(u)f &  \text{in}\, \Omega, \\
u=0 & \text{on}\ \partial \Omega,
\end{cases}\label{1f}
\end{equation}
where $A(x)$ is a symmetric function with Lipschitz continuous coefficients satisfying \eqref{lipA}.

We start recalling some known instances of  solutions in $H^1_0(\Omega)$ to problem \eqref{1f} that are already present in the literature. Consider, for simplicity,   the model case $h(s)={s^{-\gamma}}$;  as we mentioned, if $A(x)=I$ and  $f$ is an H\"older continuous function on $\overline{\Omega}$ which is bounded away from zero on $\Omega$,  then a classical solution to problem \eqref{1f} is in $\huz$ if and only if $\gamma<3$ (\cite{lm}). Switching to weak solutions with nonnegative data in $L^{m}(\Omega)$, in \cite{bo} the authors proved the existence of  a solution $u$ to \eqref{1f} in  $H^1_0(\Omega)$ if either   $\gamma=1$ and $ f\in L^1(\Omega)$ or   $\gamma<1$ and $ f\in L^{(\frac{2^*}{1-\gamma})'}(\Omega)$.
In the case $\gamma>1$ solutions are always in $H^{1}_{loc}(\Omega)$;  in \cite{am} the authors prove the existence of a solution in $\huz$ if $f\geq C>0$  is a function in $L^{m}(\Omega)$ ($m>1$) and $\gamma<\frac{3m-1}{m+1}$. See also \cite{ls, zc,cc, edr,bgh,op}, and references therein,  for further refinements and extensions. 

\medskip

Following \cite{bo,boca}, if $f\in L^{m}(\Omega)$, a function $u\in \huz$ is a  distributional solution to problem \eqref{1f} if it satisfies  $u\geq c_{\omega}>0$, for any $\omega\subset\subset \Omega$, and  
\begin{equation}\label{distri}
\displaystyle\int_{\Omega}A(x)\nabla u\cdot \nabla \varphi = \int_{\Omega} h(u)f\varphi , \ \ \ \forall \varphi \in C^1_c (\Omega)\,.
\end{equation}
Observe that in this case one has $h(u)f\in L^{1}_{loc}(\Omega)$ and  this notion  is  equivalent to the one given in Definition \ref{def} for general data provided $u\in\huz$. 
A first important  remark is that finite energy distributional  solutions to problem \eqref{1f} are in fact solutions in the usual weak sense, that is they satisfy 
\begin{equation}
		\displaystyle\int_{\Omega}A(x)\nabla u\cdot \nabla \phi = \int_{\Omega} h(u)f\phi , \ \ \ \forall \phi \in H^1_0(\Omega).
		\label{H10test}
	\end{equation}	 

Moreover, $H^{1}_{0}$-distributional solutions to our singular problem are unique provided $h$ is nonincreasing: this was already  observed in \cite{boca, cst} in the model case $h(s)={s^{-\gamma}}$ (see also \cite{ls} for some related preliminary remarks).  Indeed, we have the following result whose proof strictly follows the lines of  the one of Theorems 2.2 and 2.4 in \cite{boca} with minor modifications. For completeness we shall sketch it  in the Appendix.  
\begin{theorem}
Let $f$ in $\luo$ be a nonnegative function and let  $u\in H^1_0(\Omega)$ be a distributional solution to \eqref{1f}, then $u$ satisfies \eqref{H10test}. Moreover, if  $h$ is nonincreasing then problem \eqref{1f} admits a unique distributional solution in $\huz$.  
	\label{bocat}
\end{theorem}

In the rest of this section we then answer  the question of  whether  problem  \eqref{1f}  admits a  (unique, provided $h$ is not increasing) finite energy solution.  As we already said the existence results for problem \eqref{1f} are essentially based on an approximation scheme. Thus,  we consider  solutions $u_{n}\in \huz\cap L^{\infty}(\Omega)$ of 
\begin{equation}\label{1nn}\begin{cases}
-{\rm div} (A(x)\nabla  u_n)= h_n(u_n)f_n & \text{in}\;\ \Omega,\\
u_n=0 & \text{on}\;\ \partial\Omega,
\end{cases}\end{equation}
where $h_n:= T_n(h)$, $f_n:= T_n(f)$. One  has  (\cite{do}) that a positive constant $c_\omega$ exists such that 
\begin{equation}\label{sotto}
u_{n}\geq c_{\omega}>0, \  \ \forall  \ \omega\subset\subset\Omega\,.
\end{equation}
Our  aim will then consist in look for estimates  on  the sequence $u_{n}$ in $\huz$. 
 In order to simplify the exposition observe that it is not restrictive to assume that  $$n>\max(h(\underline{\omega}), \max_{[\underline \omega, \infty)}h(s))$$ so that we are only possibly truncating $h$ near $s=0$. \bk 

 As we will see, a major role in  the regularizing effect for the solutions is played by  the behavior of $h$ at infinity. Therefore,  in order to present the results, we will also need to assume the following:   
 \begin{equation}\label{h2}
\exists\;\ K_2,\overline{\omega}>0\ \;\text{such that}\ \;h(s) \le \frac{K_2}{s^\theta}\;\ \text{if}\ \;s>\overline{\omega},
\end{equation}
for some $\theta>0$.

\medskip 
We distinguish between the cases $\gamma\leq 1$ and $\gamma>1$ as, in the former case the presence of a possibly singular $h$ is essentially negligible, while in the latter case  a control near zero will be  needed. 
\bk

\subsection{The case $\gamma\leq 1$} 
Assuming a \emph{strong control} on $h$ at infinity then solutions to \eqref{1f} have always finite energy for any integrable data. 
\begin{theorem}\label{magg1}
	Let  $h$  satisfy \eqref{h1} and \eqref{h2}, with $\gamma\leq 1$ and $\theta\geq 1$. Then for any nonnegative $f\in\luo$ there exists a solution $u\in H^1_0(\Omega)$ to problem \eqref{1f}. 
\end{theorem}
\begin{proof}
	\noindent As we said we only  need some a priori estimates on the sequence of approximating solutions $u_n$ to \eqref{1nn} in $H^1_0(\Omega)$.  To this aim, we take $u_n$ as a test function in \eqref{1nn} and we use \eqref{h1}, and \eqref{h2} obtaining 
$$\begin{array}{l}
	\dys \alpha \int_{\Omega} |\nabla u_n|^2 \leq \int_{\Omega} h_n(u_n)f_n u_n \leq K_{1}\int_{\{u_n<\underline{\omega}\}}f_n{u_n^{1-\gamma}}+ \max_{[\underline \omega,\overline \omega]}h(s)\int_{\{\underline{\omega}\leq u_n \leq\overline{\omega}\}}f_n u_n \notag \\ \\
\dys 	+\ K_{2}\int_{\{u_n>\overline{\omega}\}}{f_{n}}{ u_n^{1-\theta}} \le K_{1}\underline{\omega}^{1-\gamma}\int_{\{u_n<\underline{\omega}\}}f + \overline{\omega} \max_{[\underline \omega,\overline \omega]}h(s)\int_{\{\underline{\omega}\leq u_n \leq\overline{\omega}\}}f + {K_{2}}{ \overline{\omega}^{(1-\theta)}}\int_{\{u_n>\overline{\omega}\}}f \le C.
\end{array}
$$
and the proof is complete.
\end{proof}

Observe that the previous proof only made use of the ellipticity condition \eqref{lipA}, on \eqref{h1}, and on \eqref{h2};  therefore, it easily extends to   more general, possibly nonlinear, operators in not necessarily smooth domains, and to the case of measure data  as  considered for instance in \cite{dc,do, op,cst}.  \bk 

\medskip 
A \emph{milder control} on $h$ at infinity (namely $\theta<1$) is not enough to ensure, in general, finite energy solutions as the following example shows. 

\begin{example}\label{luigi}
	For $N>2$ we fix $\gamma<1$ and we let $q=\frac{N(\gamma+1)}{N+2\gamma}$.  As $q<(2^{*})'=\frac{2N}{N+2}$, we can consider the positive solution $u$ to 
	\begin{equation*}
	\begin{cases}
	\displaystyle -\Delta u= f &  \text{in}\, \Omega, \\
	u=0 & \text{on}\ \partial \Omega.
	\end{cases}
	\end{equation*}
where $0\leq f\in L^{q}(\Omega)$ is such that  $u$ is $W^{1,q^{*}}_{0}(\Omega)$ but  $u\not\in H^1_0(\Omega)$. 
We have that $u$ is a distributional solution to
	\begin{equation*}
	\begin{cases}
	\displaystyle -\Delta u= \frac{g}{u^\gamma} &  \text{in}\, \Omega, \\
	u=0 & \text{on}\ \partial \Omega,
	\end{cases}
	\end{equation*}
where $g=fu^\gamma$. 	We claim that  $g$ in $L^1(\Omega)$; indeed, by H\"older's inequality one has 
	$$\dys \into fu^\gamma\le\left(\into f^q\right)^\frac{1}{q} \left(\into u^{\gamma q'}\right)^\frac{1}{q'},$$
	and, by the choice of $q$, the last integral is finite  since  $\gamma q' = q^{**}$.	
\end{example}

The previous example shows, at least in a model case, that if $f$ is asked to merely belong to  $L^1(\Omega)$ and  $\theta<1$ (here $\gamma=\theta<1$) then it is possible to find a solution $u$ not belonging to $ H^1_0(\Omega)$.   

Anyway,  also to recover the standard model in which  $\gamma=\theta<1$, the general case  $\theta>0$  can  be treat  by assuming some further requests on $f$, namely more regularity inside $\Omega$  and a control near the boundary  of the type   
	\begin{equation}
	f(x) \le \frac{c}{\delta(x)}\ \  \text{a.e.  in } \Omega_\varepsilon\,,
	\label{condfalto}
	\end{equation}
		where $\vare$ is small enough in order to guarantee that $\Omega\setminus\overline\Omega_\varepsilon$ is a smooth  subset  compactly contained in $\Omega$. 
\begin{theorem}
	Let $ 0\leq f\in L^1(\Omega) \cap L^p(\Omega\setminus\overline\Omega_\varepsilon)$ with  $p > \frac{N}{2}$ satisfying \eqref{condfalto}, and let  $h$ satisfying \eqref{h1} and \eqref{h2} with $\gamma\leq\bk1$ and $\theta>0$. Then there exists a  solution $u$ to \eqref{1f} belonging to $H^1_0(\Omega)$.
	\label{gamma<1}
\end{theorem}
\begin{proof}Without loss of generality we assume $\theta<1$ otherwise we can apply Theorem \ref{magg1} to conclude.

First of all, in view of \eqref{sotto} and on the local regularity of $f$, we can apply  classical De Giorgi-Stampacchia regularity theory to get that    $u_{n}\in C(\Omega\setminus\Omega_\varepsilon)$ and 
$$\| u_n\|_{L^{\infty}(\Omega\setminus\overline\Omega_\varepsilon)}\leq C\,, $$
where $C$  depends on $\vare$ but not on $n$. In particular, 
$$\| h_{n}(u_{n})u_n\|_{L^{\infty}(\Omega\setminus\overline\Omega_\varepsilon)}\leq C ||h(s)||_{L^{\infty}([c_{\Omega\setminus\overline\Omega_\varepsilon}, \infty))}\,.$$
   
	Therefore, testing \eqref{1nn} with $u_{n}$ and  using  both \eqref{h1} and \eqref{h2},  we have
	\begin{equation}
	\begin{array}{l}	\label{stimaungammaminore}
	\alpha \displaystyle \int_{\Omega} |\nabla u_n|^2 \leq  \int_{\Omega} h_n(u_n)f_nu_n \le K_{1}\int_{\Omega_\varepsilon\cap \{u_n< \underline{\omega}\}}f_n u_n^{1-\gamma} + \int_{\Omega_\varepsilon\cap \{\underline{\omega} \le u_n \le \overline{\omega} \}}h_n(u_n)f_nu_n  
	\\\\ \dys 
	+ K_{2}\int_{\Omega_\varepsilon\cap \{u_n> \overline{\omega}\}}f_n u_n^{1-\theta} + \int_{\Omega\setminus\overline\Omega_\varepsilon}  h_n(u_n)f_nu_n \le  K_{1}\underline{\omega}^{1-\gamma}\int_{\Omega_\varepsilon\cap \{u_n< \underline{\omega}\}}f
	\\\\ 	\dys 
	+
	\overline{\omega}  \max_{[\underline \omega,\overline \omega]}h(s)\bk \int_{\Omega_\varepsilon\cap \{\underline{\omega} \le u_n \le \overline{\omega} \}}f + K_{2} \int_{\Omega_\varepsilon\cap \{u_n> \underline{\omega}\}}f_nu_n^{1-\theta}  \\\\ \dys  
	+ C ||h(s)||_{L^{\infty}([c_{\Omega\setminus\overline\Omega_\varepsilon}, \infty))} \int_{\Omega\setminus\overline\Omega_\varepsilon}  f\,.\end{array}
	\end{equation}

What is needed to conclude is then the control of the term 
\begin{equation}\label{coming}
	  \int_{\Omega_\varepsilon\cap \{u_n> \underline{\omega}\}}f_nu_n^{1-\theta} \,. 
\end{equation}

\medskip

Consider the smooth solutions $w_{n}$ to  the auxiliary problem
\begin{equation}
		\begin{cases}
		\displaystyle -\operatorname{div}(A(x)\nabla w_n) = \overline{h}_{n}(w_n)f_n &  \text{in}\, \Omega, \\
		w_{n}=0 & \text{on}\ \partial \Omega,
		\label{aux}
		\end{cases}
\end{equation}
where $\overline{h}_{n}(s)$ is the truncation at level $n$ of the function $\overline{h}$ defined in \eqref{oh}. By comparison $u_{n}$ is a sub-solution to problem \eqref{aux} and so $u_{n}\leq w_{n}$ for any $n$. 
 We now look for a super-solution to problem \eqref{aux} in $\Omega_{\vare}$ of the form  $M\varphi_1^t$, for some $M, t>0$,  in order to get,  by comparison (see for instance \cite[Theorem 10.7]{gt} and the discussion at the end of its proof\bk),  that 
	\begin{equation}
	M\varphi_1^t\ge w_{n}\geq  u_n, \ \ \ \text{in } \Omega_\varepsilon.
	\label{sopragamma<1}
	\end{equation}
	
	 We fix $t=\frac{1}{\gamma+1}$. 
	   We need  to check that  the first inequality in 
	\begin{equation}\nonumber
	\begin{array}{l}
\dys 	-{\rm div} (A(x)\nabla M\varphi_1^t)=\overline{h}(M\varphi_1^t)\left(M t(1-t)\frac{\varphi_1^{t-2}}{\overline{h}(M\varphi_1^t)}A(x)\nabla \varphi_1\cdot\nabla \varphi_1  + Mt\lambda_1\frac{\varphi_1^{t}}{\overline{h}(M\varphi_1^t)}\right ) \\ \\ \dys\ge f\overline{h}(M\varphi_1^t)\ge f_n\overline{h}_n(M\varphi_1^t)  
	\label{soprasolgammaminore}
	\end{array}
	\end{equation}
	holds in $\Omega_\varepsilon$. 
Dropping a positive term and recalling both \eqref{hopfvarphi2} and  \eqref{condfalto} the previous is implied by  the first inequality in 
	\begin{equation*}
	\alpha M t(1-t)\frac{\varphi_1^{t-2}}{\overline{h}(M\varphi_1^t)}|\nabla \varphi_1|^{2} \ge \frac{c_2 c}{\varphi_1} \ge \frac{c}{\delta} \ge f\ \text{ in } \Omega_\varepsilon,
	\label{gammaminore4}
	\end{equation*}
that, in view of the Hopf's lemma,  essentially reduces (up to normalization of not relevant constants) in proving that there exists a positive constant $M$ such that 
\begin{equation}\label{maxo}
\overline{h}(M\varphi_1^{t})(M\varphi_1^t)^\gamma\leq M^{1+\gamma} \ \ \ \text{a.e. in}\ \   \Omega_\varepsilon.
\end{equation}
We have
$$ 
\overline{h}(M\varphi_1^{t})(M\varphi_1^t)^\gamma\leq\max (K_1,\max_{[\rho,\infty)}\overline{h}(s) (M\varphi_1^t)^\gamma)\,,
$$
so that  \eqref{maxo} is satisfied up the following  choice 
$$
 M\geq \max(K_1^{t},\max_{[\rho,\infty)}\overline{h}(s) \|\varphi_1^t\|_{L^\infty(\Omega_\vare)}^{\gamma})\,.
$$	
  Observe that, by standard regularity theory,   $\varphi_1^t$  is continuous up to the boundary and smooth inside $\Omega$. \bk By possibly increase the value of $M$  we also assume
	\begin{equation*}
	M\varphi_1^t\ge w_{n}\geq u_n \ \  \text{  in  } \partial (\Omega\setminus\partial\Omega_\varepsilon), 
	\end{equation*}
and we can apply the comparison principle in $\Omega_\vare$  obtaining \eqref{sopragamma<1}. 

Therefore, coming back to \eqref{coming}, using also \eqref{condfalto} and \eqref{hopfvarphi2}, we have 
	\begin{equation*}
	\displaystyle \int_{\Omega_\varepsilon\cap \{u_n> \overline{\omega}\}}f_nu_n^{1-\theta}  \le M^{1-\theta}\int_{\Omega_\varepsilon}\frac{c}{\delta}\varphi_1^{\frac{1-\theta}{1+\gamma}}\le C \int_{\Omega_\varepsilon}{\delta^{-\frac{\theta+\gamma}{1+\gamma}}}<\infty,
	\label{primogammaminore2} 
	\end{equation*}
	 since  $\theta<1$,  and the proof is complete recalling \eqref{stimaungammaminore}.

\end{proof}

\begin{remark}
 Although assumption  \eqref{condfalto} is not too restrictive, it seems to be only technical and it could be removed. On the other hand some stronger local summability inside $\Omega$ seems to be needed (compare also with Example \ref{ex3} later). 
\end{remark}

\subsection{The case $\gamma>1$}
\label{regularity}

Unless for the case of compactly supported data (see the discussion following the proof of Theorem \ref{34} below),  it is known  that boundedness away from zero of the datum $f$ (at least near the boundary)  is very useful in order to obtain sharp regularity result as the one we look for (\cite{lm,am,zc}). For merely nonnegative data we have the following:
\begin{theorem}
	\label{34} Let $ f$ be a nonnegative  function in $L^m(\Omega)$ with $m> 1$, and let $h$ satisfies    \eqref{h1} and \eqref{h2}, with $\theta\geq 1$.  Then there exists a  solution $u$ to \eqref{1f} that belongs to $H^1_0(\Omega)$ provided
	$$\displaystyle 1<\gamma<2 - \frac{1}{m}\,.$$   
\end{theorem}
\begin{proof}
	We take $u_n$ as a test function in \eqref{1nn}; we have 
	\begin{equation}\label{313}
	\begin{array}{l}
	\displaystyle \alpha\int_{\Omega}|\nabla u_n|^2 \le \int_{\Omega}h_n(u_n)f_nu_n \le K_{1}\int_{\{u_n< \underline{\omega}\}}f_n u_n^{1-\gamma} + \int_{\{\underline{\omega}\le u_n \le \overline{\omega}\}}h_n(u_n) f_n u_n  \\ \\ \dys 
	+ K_{2}\int_{\{u_n> \overline{\omega}\}}f_n u_n^{1-\theta} \le  K_{1}\int_{\{u_n< \underline{\omega}\}}f_n u_n^{1-\gamma}  +  \overline{\omega}\max_{[\underline \omega,\overline\omega]}{h}(s) \int_{\{\underline{\omega}\le u_n \le \overline{\omega}\}}f  \bk + K_{2}\int_{\{u_n> \overline{\omega}\}}f \overline{\omega}^{1-\theta}.
	\end{array}
	\end{equation}
 In order to conclude  we need to estimate the first term on the right hand side of \eqref{313}. 
	 To do that, we consider the nonincreasing and continuous function $\underline{h}:\mathbb{R}^+\rightarrow\mathbb{R}^{+}$ given by \eqref{uoh}. Recall that
		 \begin{equation}
		 	\underline{h}(s)\le h_n(s)\,,\ \ \forall \ s>0, \ \ n\in\mathbb{N}\,.
		 	\label{funzionesotto}
		 \end{equation}
		 We then  consider $v_n\in H^1_0(\Omega) $ as the solutions to 
$$
		\begin{cases}
		\displaystyle -\operatorname{div}(A(x)\nabla v_n) = \underline{h}(v_n)f_n &  \text{in}\, \Omega, \\
		v_{n}=0 & \text{on}\ \partial \Omega.
		\label{pbv}
		\end{cases}
$$
As $\underline{h}$  is nonincreasing, it is known that $v_{n}$ is nondecreasing  with respect to $n$ (see \cite{bo,do}) and, using \eqref{funzionesotto}, by comparison $u_n\ge v_n\ge v_{1}$. 
	 We  can  apply Hopf's lemma to $v_{1}$ \bk (see for instance \cite{edr,sdl}) to deduce that \bk
	\begin{equation}
	v_{1}(x)\ge C\delta(x), \text{  for  } x \in \Omega.
	\label{stimau1}
	\end{equation}\bk
	Thus, it follows from the H\"older inequality, from \eqref{stimau1} and from the fact that $u_n\ge v_{1}$, that for the first term in the left hand side of \eqref{313} we have
	$$\int_{\{u_n< \underline{\omega}\}}f_n u_n^{1-\gamma} \le \int_{\Omega}fv_{1}^{1-\gamma}\le C^{1-\gamma}||f||_{L^m(\Omega)}\left( \int_{\Omega} \frac{1}{\delta^{(\gamma-1)m'}}\right)^{\frac{1}{m'}}<\infty,$$
	since $\displaystyle \gamma<2 - \frac{1}{m}.$   
 This concludes the proof.
\end{proof}

\begin{remark}
Observe that the previous proof works  also for $\gamma=1$,  in particular,  as $m\to 1^{+}$, one  recovers the case in which finite energy solutions always exist for $f\in\luo$ in  \emph{continuity} with the result of Theorem \ref{magg1}. 
\end{remark}
\bk
\medskip

Though, in this generality, it seems not so easy to improve it, the (upper) threshold  on $\gamma$ given in Theorem \ref{34}  turns out to be not the optimal one. Consider, for instance, the model case $h(s)={s^{-\gamma}}$ and a function $f$ in $L^{m}(\Omega)$ with $m>1$. In \cite{am} the authors prove the existence of a solution $u\in H^{1}_{0}(\Omega)$ if $1<\gamma<\frac{3m-1}{m+1}$ provided $f$ is strictly bounded away from zero.  Prototypical examples show however that finite energy solutions can be found up to $\gamma<3-\frac{2}{m}$ (see Example \ref{ex} below). Observe that, $3-\frac{2}{m}>\frac{3m-1}{m+1}$ and that, as $m$ tends to infinity, one formally recovers the Lazer-McKenna threshold $\gamma=3$ for bounded (also away from zero, and smooth) data. 

Apart from explicit examples, we also refer to \cite{zc,bgh} in which, as already suggested in \cite[Section 4]{lm}, one sees that these limit values can be reached in the case of the laplacian and  a smooth and bounded away from zero datum $f$ that blows up uniformly at $\partial\Omega$ at a precise rate.   Let us only mention the opposite case of $f$ having compact support on $\Omega$. If this is the  case in fact, if  $\gamma,\theta\geq 1$ then, in view of \eqref{sotto},  $h_{n}(u_{n})u_{n}$ is uniformly bounded on the support of $f$ and the estimate on $u_{n}$ in $\huz$ is for free for any $f\in L^{1}(\Omega)$.  \bk

\medskip

Let us consider
\begin{equation}\label{1cin}\begin{cases}
\dys -{\rm div} (A(x)\nabla  u)=\frac{f}{u^{\gamma}} & \text{in}\ \;\Omega,\\
u=0 & \text{on}\ \;\partial\Omega,
\end{cases}\end{equation}
for $\gamma>1$ and $f\in L^{m}(\Omega)$, $m>1$, a positive function. We have the following 
\begin{theorem}
	Let $m, \gamma>1$. Then the solution $u$ to problem \eqref{1cin} belongs to $H^1_0(\Omega)$ for any positive $f\in L^{m}(\Omega)$ if and only if $\gamma<3-\frac{2}{m}$. 
	\label{teosharp}
\end{theorem}

The sufficient condition in Theorem \ref{teosharp} is the  easiest part as it  essentially follows by a (highly nontrivial) result proved, with variational tools,  in  \cite{sz}. More precisely, a line by line re-adaptation to the case a bounded matrix of the proof of \cite[Theorem 1]{sz} allows us to state the following
\begin{theorem}\label{cino}
	Let  $\gamma>1$ and let $f$ be a positive function in  $L^1(\Omega)$. 
	Then there exists a unique solution $u\in H^1_0(\Omega)$ to problem \eqref{1cin} if and only if there exists a function $u_0\in H^1_0(\Omega)$ such that 
	\begin{equation}
	\int_{\Omega}fu_0^{1-\gamma}<+\infty.
	\label{condcinese}
	\end{equation}
\end{theorem}

\begin{proof}[Proof of Theorem \ref{teosharp}]
 To prove that the solution $u$ belongs to $H^1_0(\Omega)$ only plug 
  $u_{0}(x)=\delta(x)^{t}$ into \eqref{condcinese}. Using H\"older inequality one has
	$$
	\int_{\Omega}fu_0^{1-\gamma}\leq  C\left(\into \delta^{t(1-\gamma)m'}\right)^{\frac{1}{m'}}\,.
	$$ 
Since $\gamma<3-\frac{2}{m}$ it is possible to choose $t$ greater than $\frac12$ so that $t(1-\gamma)m'>-1$ and one can apply the result of Theorem \ref{cino}. 
	
	\noindent In order to prove optimality we let $\gamma\geq 3-\frac{2}{m}$, we define
	\begin{equation}
	f(x):=  \max\left(\frac{1}{\delta(x)^{\frac{1}{m}}\log\left(\frac{1}{\delta(x)}\right)}, 1\right)\,,
	\label{condizionef}
	\end{equation}
and $f_{n}=T_{n}(f)$. One can prove that a suitable positive constant  $M$ exists such that $M\varphi_{1}^{t}$ with 
\begin{equation*}
	\displaystyle t = \frac{2}{\gamma+1} - \frac{1}{m(\gamma+1)}\,,
	\end{equation*}
	(observe that $0<t\leq \frac12$) is a  strong \bk super-solution of the approximating problems 
	\begin{equation}\label{1cinn}\nonumber\begin{cases}
-{\rm div} (A(x)\nabla  u_{n})=\frac{f_{n}}{(u_{n}+\frac{1}{n})^{\gamma}} & \text{in}\;\ \Omega,\\
u_{n}=0 & \text{on}\;\ \partial\Omega.
\end{cases}\end{equation}

	Indeed, since 
	$$\frac{f}{M^\gamma\varphi_1^{t\gamma}} \ge \frac{f_n}{(M\varphi_1^{t}+\frac{1}{n})^\gamma}\,,$$
	 we only need to show that 
	$$\displaystyle -{\rm div}(A(x) \nabla M\varphi_1^t )= \frac{1}{M^\gamma\varphi_1^{\gamma t}}\left(M^{1+\gamma}t(1-t)A(x)\nabla \varphi_1\cdot\nabla\varphi_{1}\varphi_1^{t-2+\gamma t} + M^{1+\gamma}\lambda_1t\varphi_1^{t+\gamma t}\right)\ge \frac{f}{M^\gamma\varphi_1^{\gamma t}},$$
	that is implied by 
	\begin{equation}
	\alpha M^{1+\gamma}t(1-t)|\nabla \varphi_1|^2\varphi_1^{t-2+\gamma t} + M^{1+\gamma}\lambda_1t\varphi_1^{t+\gamma t}\ge f.
	\label{stimaesempio1}
	\end{equation}
  Let $\vare< {e^{-1}}$ be a small enough positive number such that $\Omega\setminus\overline\Omega_\epsilon$ is smooth and compactly contained in $\Omega$ and  observe that both terms on the left hand side of \eqref{stimaesempio1} are nonnegative.  \bk 
	Therefore, as $\varphi_{1}$ is locally strictly away from zero,   it is possible to choose $M$ large enough such that  $M^{1+\gamma}\lambda_1t\varphi_1^{t+\gamma t}\ge f $\,, for any  $x \in \Omega\setminus\overline\Omega_\epsilon$. \bk 
	
	Otherwise, if $x \in \Omega_\epsilon$, it suffices to prove the first inequality in 
	\begin{equation*}
	M^{1+\gamma}t(1-t)|\nabla \varphi_1|^2\varphi_1^{t-2+\gamma t}\ge \frac{1}{\delta(x)^{\frac{1}{m}}}\ge \frac{1}{\delta(x)^{\frac{1}{m}}\log\left(\frac{1}{\delta(x)}\right)} = f(x),
	\label{stimaesempio3}
	\end{equation*}
	where we used that  $\delta(x)<\frac{1}{e}$.
	Using Hopf's lemma (and \eqref{hopfvarphi2}) the previous reduces by 
	\begin{equation}
 \frac{M^{1+\gamma}}{\delta(x)^{-t+2-\gamma t}} \ge \frac{K}{\delta(x)^\frac{1}{m}}\,,
	\label{stimaesempio4}
	\end{equation}
	where $K$ is a constant that only depends on $\Omega, \varphi_{1}$, $m$, and $\gamma$. Thanks to the choice of 
	$t$  \eqref{stimaesempio4} is satisfied for $M$ large enough. 
	
Therefore, we can apply the comparison principle  between  $M\varphi_1^t$ and $u_n$ (as in the proof of Theorem \ref{gamma<1})  to obtain $M\varphi_1^t\ge u_n$ and so $ M\varphi_1^t\ge u$ passing to the  a.e. limit.
	 \noindent Now suppose by contradiction that $u\in H^1_0(\Omega)$.  Then, by Theorem \ref{bocat},  we can use $u$ as test function in \eqref{1cin}; recalling \eqref{hopfvarphi2}, \eqref{condizionef},  and that  $\gamma \ge 3-\frac{2}{m}$,  we then have
	$$\beta\bk \displaystyle \int_{\Omega} |\nabla u|^2 \ge \int_{\Omega}{f}{u^{1-\gamma}} \ge M^{{t(1-\gamma)}}\int_{\Omega}f\varphi_1^{t(1-\gamma)}=\infty,$$
	 a contradiction. 
\end{proof}

\subsection{On the summability of  the lower order term} \label{furt} 

We discuss a  prototypical  example of solutions to problems as in \eqref{1cin}. 
\begin{example}\label{ex}
Consider problem \eqref{1cin}  with $A\equiv I$ (i.e. the case of the laplacian). Let $\Omega=B_{1}(0)$ and $u = (1-|x|^{2})^{\eta}$, with $\eta>0$.  Then, if 
	\begin{equation}\label{al}\frac{1}{1+\gamma}<\eta < 1,\end{equation}  $u$ solves \eqref{1cin} with 
	$$
	f\sim \frac{1}{(1-|x|^{2})^{2-\eta -\eta\gamma}}\in L^{1}(\Omega). 
	$$

First of all observe that, as $\eta<1$, then  $-\Delta u\notin L^{1}(\Omega)$, while,   $-\Delta u \in L^{1}(\Omega, \delta)$ for any $\eta>0$. This latter remark should be compared with  Lemma \ref{M1respectdelta} below. 
Using \eqref{al}, we also have
\begin{itemize}
	\item [(i)] if $\gamma=1$ the solution is always in $H^1_0(\Omega)$ as expected,
	\item [(ii)] if $\gamma>1$ then $f$ is in $L^{m}(\Omega)$ provided 
	$$
	\eta >\frac{2-\frac{1}{m}}{\gamma+1}\,,
	$$
	In particular, $u\in H^1_0(\Omega)$ if $\gamma<3-\frac{2}{m}$. 
	\item [(iii)] if $\gamma<1$ then $u$ is always in $H^{1}_{0}(\Omega)$  and $f\in L^{m}(\Omega)$ for any
	$$
	m<\frac{1}{2-\eta-\eta\gamma}\,.
	$$
	We observe that 
	\begin{equation}\label{opt}
	\frac{1}{2-\eta-\eta\gamma}\nearrow\frac{1}{1-\gamma}\ \ \text{as $\eta\to1^{-}$}\,.
	\end{equation}
\end{itemize}
\end{example}\bk 
The previous example shows that one cannot expect in general the lower order term ${f}{u^{-\gamma}}$ to belong to $L^{1}(\Omega)$, even for small $\gamma$. Suitable  weighted summability of the lower order term  will be given in the next section  (see Lemma \ref{M1respectdelta}).  Although observe that in (iii) $m$ needs to be small enough.   For $m=\infty$ (i.e. $f$ is a  nonnegative bounded function) and a mild blow up at $0$ (i.e. $\gamma<1$), as a consequence of Theorem 1.2 in \cite{gl},  one has ${f}{u^{-\gamma}}\in L^{1}(\Omega)$. \bk 
In general,  we have the following result in which  the threshold $\frac{1}{1-\gamma}$ cannot be improved in view of \eqref{opt} and   one can allow $h(\infty)\not= 0$.  \bk

\begin{theorem}
	Let $h$ satisfy \eqref{h1} with $\gamma<1$, and let $f$ be a nonnegative function in $L^{m}(\Omega)$ with $m>\frac{1}{1-\gamma}$. Then there exists a solution to problem \eqref{1f} such that  $h(u)f \in L^{1}(\Omega)$. 
	\label{gamma<1L1}
\end{theorem}
\begin{proof}

Let 
  $v_{n}$ be the solution to problem 
$$
		\begin{cases}
		\displaystyle -\operatorname{div}(A(x)\nabla v_n) = \underline{h}(v_n)f_n &  \text{in}\, \Omega, \\
		v_{n}=0 & \text{on}\ \partial \Omega,
		\end{cases}
$$
defined as in the proof of Theorem \ref{regularity}. Notice that the fact that $u_{n}\geq v_{n}\geq v_{1}\geq C \delta$ in $\Omega$ (for suitable constant $C$)  is independent of the value of $\gamma$ and so it holds true still in the case $\gamma<1$.  Passing to the a.e. limit  one then obtains $u\geq v_{1}\geq C \delta$ and so, by H\"older's inequality, one has 
$$\begin{array}{l}
\dys 	\into h(u) f \le K_{1} \int_{\{u< \underline{\omega}\}}  f u^{-\gamma} +\sup_{[\underline \omega,\infty)}{h}(s)\int_{\{u \ge \underline{\omega}\}}  f
	\\ \\ \dys  \leq  
	C\into f v_{1}^{-\gamma} + \sup_{[\underline \omega,\infty)}{h}(s)\int_{\Omega}  f  \leq C\left(\into v_{1}^{-m'\gamma}\right)^\frac{1}{m'}+C\,\end{array}
$$
	and the last integral is finite as $m> \frac{1}{1-\gamma}$.
\end{proof}

\medskip 

\section{Infinite energy solutions} 
\setcounter{equation}{0}	
\label{uniqueness_distributional}

\subsection{Existence of a solution} In this section we prove Theorem \ref{esiste}.

\begin{proof}[Proof of Theorem \ref{esiste}]
Following   \cite{do},  the existence of a  solution to \eqref{1} can be  proven by approximating with the solutions to the desingularized problems 
\begin{equation}\label{1n}\begin{cases}
L u_n= h_n(u_n)f_n & \text{in}\;\ \Omega,\\
u_n=0 & \text{on}\;\ \partial\Omega,
\end{cases}\end{equation}
where $h_{n}(s)=T_{n}(h(s))$ and $f_{n}$	is  a sequence of smooth functions suitably converging to $\mu$. The following a priori estimates holds:
$$
\|T_k(u_{n})\|_{H^1_0(\Omega)}\leq C\;\ \text{if}\ \;\gamma\le1\,,\ \text{and }\ \ 
	\|T_k(u_{n})^{\frac{\gamma+1}{2}}\|_{H^1_0(\Omega)}\leq C\;\ \text{if}\;\ \gamma>1\,,
$$
for any $k>0$. Moreover, \eqref{sotto} holds 
and, up to subsequences,  $u_{n}$ a.e. converges towards a positive function $u\in W^{1,1}_{loc}(\Omega)$ such that
\begin{equation}\label{trunc}\begin{cases}T_k(u)\; \ \ \ \in H^1_0(\Omega)\;\ \text{if}\ \;\gamma\le1\\
	T_k^{\frac{\gamma+1}{2}}(u)\;\ \in H^1_0(\Omega)\;\ \text{if}\;\ \gamma>1\,,
	\end{cases} \end{equation}
$ h(u) \in L^\infty_{loc}(\Omega,\mu_d)$, and \eqref{weakdef3}	is satisfied (see \cite[Theorem 3.3]{do}).  One also has that, for any $\gamma>0$ and for any $k>0$, 
\begin{equation}\label{gk} \|G_k(u_{n})\|_{W^{1,q}_{0}(\Omega)}\leq C\,, \ \ \text{for all}\ \ q<\frac{N}{N-1} \,,\end{equation}
where $C$ does not depend on $n$; in particular $G_{k}(u)\in W^{1,1}_{0}(\Omega)$. 	 The proof of \eqref{gk} is  standard and it amounts to choose $T_{r}(G_{k}(u_{n}))$  as test function in \eqref{1n}; this leads to the following estimate 
	$$\dys \into |\nabla T_r(G_k(u_n))|^2\le Cr,\ \ \ \text{for any}\ r>0\,,$$
which is known to imply \eqref{gk}. 
So that, recalling \eqref{trunc}, if $\gamma\leq 1$ we have that $u\in W^{1,1}_{0}(\Omega)$ and the proof is complete. 

If $\gamma>1$ As $u = T_1(u) +G_1(u)$ this implies in particular that $u\in\luo$ and, moreover, using H\"older's inequality one has 
$$
\begin{array}{l}
\dys \frac{1}{\vare}\int_{\Omega_{\vare}} u \leq \vare^{\frac{1-\gamma}{\gamma+1}}\left( \frac{1}{\vare}\int_{\Omega_{\vare}} T_{1}(u)^{\frac{\gamma+1}{2}}\right)^{\frac{2}{\gamma+1}}|\Omega_{\vare}|^{\frac{\gamma-1}{\gamma+1}} + \frac{1}{\vare} \int_{\Omega_{\vare}} G_{1}(u) \\ \\
\dys \leq C\left(\frac{1}{\vare}\int_{\Omega_{\vare}} T_{1}(u)^{\frac{\gamma+1}{2}}\right)^{\frac{2}{\gamma+1}}+  \frac{1}{\vare}\int_{\Omega_{\vare}} G_{1}(u)\stackrel{{\vare}\to0}{\longrightarrow} 0
\end{array}
$$
in view of \eqref{trunc}, that is  \eqref{bd} holds. 

\end{proof}

\begin{remark}
We stress that, as can be deduced by the previous proof,   all weak solutions in the sense considered, for instance,  in  \cite{bo,po,op,cst} also satisfy problem \eqref{1} in the sense of Definition \ref{def}.  We also remark that the previous proof easily extends to the case of a matrix $A(x)$ with bounded measurable coefficients in Lipschitz domains. 
\end{remark}

\subsection{Uniqueness of  solutions} 
We now   prove   Theorem \ref{uniquenessmain}. The proof will rely on  an application of the Kato's type inequality given in Lemma \ref{Katoinequalitylocal}.  

  The following is  a suitable extension  of a  property which is well known for  superharmonic functions (see for instance \cite{cm,v, mv} and references therein).  It will be a key ingredient in order to prove suitable   weighted summability on the lower order term of our singular problem.    \bk

\begin{lemma}\label{M1respectdelta}
Let $u\in L^1(\Omega) \cap W^{1,1}_{loc}(\Omega)$ be such that $Lu = \mu$ in the sense of distributions for some nonnegative measure $\mu\in \mathcal{M}_{loc}(\Omega)$. Then $\mu\in \mathcal{M}(\Omega,\delta)$.	
\end{lemma}
\begin{proof}
Let $\Phi$  be a convex smooth function  with bounded $\Phi'$ \bk and which vanishes in a neighborhood of $0$. For instance,  for a fixed $k>0$, a good choice is to consider $\Phi$ as  a convex smooth  function that agrees with  $|G_{k}(s)|$ for every $s$ but $|s|\in [k,2k]$\bk. Moreover, let $\xi$ be the solution to 
\begin{equation}\label{xi}
\begin{cases}
\displaystyle L^{*} \xi = 1 &  \text{in}\, \Omega, \\
\xi=0 & \text{on}\ \partial \Omega.
\end{cases}
\end{equation}

\noindent In order to conclude  it suffices to prove that 
$$
\into \xi\mu \leq C,
$$
as, by Hopf's lemma, one can deduce that   $\xi\geq c\delta$ on $\Omega$. 

Let $\varphi_{n}:= \frac{1}{n}{\Phi (n\xi(x))}$. It is easy to check that $\varphi_{n}$ has compact support in $\Omega$ and that $\varphi_{n}$ converges to $\xi$ a.e. in $\Omega$. Using \eqref{lipA} and the convexity of $\Phi$, we have, in the sense of distributions  
$$
L^{*} \varphi_{n}= \Phi'(n\xi)L^{*} \xi - A^*(x)\nabla \xi\cdot \nabla \xi \Phi''(n\xi) \leq \Phi'(n\xi).
$$

Then
$$
\into \varphi_{n}\mu \le \into  u L^{*}\varphi_{n}\ \leq \|\Phi'\|_{\infty}\into u\leq C,
$$
and the proof is complete using the Fatou lemma. 
\end{proof}

\begin{remark} First important remark is that Lemma \ref{M1respectdelta} applies to distributional solutions to our problem \eqref{1} yielding that both $h(u)\mu_{d}$ and $\mu_{c}$ belongs to $\mathcal{M}(\Omega,\delta)$. Observe that this  fact is quite general and, for instance, it does  not require any monotonicity on $h$.  Also remark that, if $\mu\in L^{1}(\Omega)$ then the same argument shows that $h(u)\mu\in L^{1}(\Omega, \delta)$ (see also  \cite{edr} where analogous  estimates were proven in the case $\mu\in L^{\infty}(\Omega)$ and $h(s)={s^{-\gamma}}$).   
\end{remark}

We are now in the position to prove Theorem \ref{uniquenessmain}. We  stress  that no control on the function $h$ at infinity nor at zero  is  required for the following proof to hold. 
\begin{proof}[Proof of Theorem \ref{uniquenessmain}]
	It follows from Lemma \ref{M1respectdelta} that $h(u)\mu_d + h(\infty)\mu_c$ belongs to $\mathcal{M}(\Omega,\delta)$. In particular being $\mu_d$ a diffuse measure and $h(u)\delta$ measurable with respect to $\mu_d$ then $h(u)\mu_d$ is itself a diffuse measure in $\mathcal{M}(\Omega,\delta)$. 
	\noindent \\We consider a distributional solution $u$ in the sense of Definition \ref{def}, that  is it satisfies 
		\begin{equation*} 
		\displaystyle \int_{\Omega}A(x)\nabla u \cdot \nabla \varphi =\int_{\Omega} h(u)\varphi\mu_d + h(\infty)\int_{\Omega} \varphi\mu_c, \ \ \ \forall \varphi \in C^1_c(\Omega).
		\end{equation*}
	First step is observing that taking in the previous $\varphi= \eta_{k} \phi$ where $ \phi \in C^1_0(\overline{\Omega})$ such that $L^{*}\phi\in L^{\infty}(\Omega)$,  \bk  and $\eta_{k}\in C^{\infty}_{c}(\Omega)$ is such that $0\le\eta_{k}\le1$, $\eta_{k}=1$ when $\delta(x) > \frac{1}{k}$, $||\nabla \eta_{k}||_{L^\infty(\Omega)}\le k$,  and $||L^{*} \eta_{k}||_{L^\infty(\Omega)}\le Ck^2$, we have
		\begin{equation}\label{uniquenessteo1} 
		\displaystyle \int_{\Omega}A(x)\nabla u \cdot \nabla (\eta_{k} \phi) =\int_{\Omega} h(u) \eta_{k} \phi\mu_d + h(\infty)\int_{\Omega}  \eta_{k} \phi\mu_c. 
		\end{equation}
	We want to pass to the limit in $k$ the previous formulation.  Since $h(u)\mu_d + h(\infty)\mu_c$ belongs to $\mathcal{M}(\Omega,\delta)$ and $\phi \in C^1_0(\overline{\Omega})$ we use dominated convergence theorem to pass to the limit on the right hand side of \eqref{uniquenessteo1}. 	
	Concerning the left hand side of \eqref{uniquenessteo1} we have
		\begin{equation}\label{uniquenessteo1lhs} 
		\displaystyle \int_{\Omega}A(x)\nabla u \cdot \nabla (\eta_{k} \phi) = - \int_{\Omega}A^{*}(x) \nabla \eta_{k}\cdot \nabla \phi u + \int_{\Omega}u\eta_{k} L^{*}\phi - \int_{\Omega}A^{*}(x) \nabla \phi\cdot \nabla \eta_{k} u + \int_{\Omega}u\phi L^{*} \eta_{k}. 
		\end{equation}	
		Since $L^{*}\phi\in L^{\infty}(\Omega)$, again by dominated convergence we have 
		$$
		\int_{\Omega}u\eta_{k} L^{*}\phi \stackrel{k\to\infty}{\longrightarrow} \int_{\Omega}u L^{*}\phi\,.
		$$
	It follows from  \eqref{bd} that 
		\begin{equation*}
		\displaystyle \lim_{k\to\infty}\int_{\Omega}|A^{*}(x) \nabla \eta_{k}\cdot \nabla \phi u| \le  \lim_{k\to\infty}\beta Ck\int_{\{\delta< \frac{1}{k}\}}u=0,
		\end{equation*}
so that the first and the third term on the right hand side of  \eqref{uniquenessteo1lhs}  vanish, while the last one  can be treated as follows
		\begin{equation*}
		\displaystyle\lim_{k\to\infty}\int_{\Omega}|u\phi L^{*} \eta_{k}| \le \lim_{k\to\infty}\beta Ck^2\int_{\{\delta< \frac{1}{k}\}}u|\phi|\le \lim_{k\to\infty}\beta Ck\int_{\{\delta< \frac{1}{k}\}}u = 0.
		\end{equation*}
	Collecting the previous  we deduce that
		\begin{equation}\label{udistribuzionaleveron} 
		\displaystyle  \int_{\Omega}u L^{*} \phi =\int_{\Omega} h(u) \phi\mu_d + h(\infty)\int_{\Omega} \phi\mu_c, \ \ \ \forall \phi \in C^1_0(\overline{\Omega}). 
		\end{equation}\bk

\medskip 
	 				
Now, by applying \eqref{udistribuzionaleveron} to the difference of two distributional solutions  $v$ and  $w$ of \eqref{1} we have
	 	 		\begin{equation*}
	 	 		\displaystyle \int_{\Omega} (v-w) L^{*} \xi =\int_{\Omega} (h(v)-h(w)) \xi \mu_d,
	 	 		\end{equation*}	
	  for every $\xi \in C^1_0(\overline{\Omega})$ such that $L^{*} \xi\in L^\infty(\Omega)$. 	 		
	As $(h(v)-h(w)) \mu_d$  is a diffuse measure in $\mathcal{M}(\Omega,\delta)$,  we can apply Lemma \ref{Katoinequalitylocal} to obtain
			\begin{equation*}
			\displaystyle  \int_{\Omega} (v-w)^+ L^{*} \varphi \le \int_{\{v\ge w\}} (h(v)-h(w))\varphi \mu_d\le 0,
			\end{equation*}	
	for every $0\le\varphi\in C^1_c(\Omega)$ such that $L^{*} \varphi \in L^\infty(\Omega)$. Now  we reason  again as in the proof of  \eqref{udistribuzionaleveron}; we consider $\xi\eta_{k}$ where $\eta_{k}$ is defined as before and $\xi$ is defined by \eqref{xi}, and we deduce that 
				\begin{equation*}
			\displaystyle  \int_{\Omega} (v-w)^+ L^{*} \xi \le 0\,,
			\end{equation*}
		that,	by definition of   $\xi$, allows us to conclude that $v\le w$. The proof is completed by interchanging the roles of $v$ and $w$.
\end{proof}

\section{Some  further remarks, extensions, and open problems}\label{furte}
\setcounter{equation}{0}
In Section   \ref{finite} we analyzed some instances of   finite energy solutions to problem    \eqref{1f} depending on both the regularity of the datum $f$ and on the behavior of the nonlinearity $h$ both at zero and at infinity.  Observe that, the strongly singular case  at infinity (i. e.  $\theta <1$) of Theorem \ref{34},  can not be treated  in general.  The pathological phenomenon shown by  Example \ref{luigi}, in fact, is essentially due to the behavior at infinity of $h$; this fact  is highlighted by the following suitable extension  of this example. 

 \begin{example}\label{ex3}
For $N>2$ and  a fixed $\theta<1$,  we choose  two parameters $m$ and $q$ such that 
$$
1\leq m<q<\frac{2N}{N+2}\, \ \ \ \text{and}\ \ \ 0<\theta\leq \frac{N(q-m)}{m(N-2q)}\,.  
$$
Notice that this choice is always possible since the function $ \nu(m,q)=\frac{N(q-m)}{m(N-2q)}$ is continuous around $(m,q)=(1,\frac{2N}{N+2})$ and  $ \nu(1,\frac{2N}{N+2})=1^{}$. 
Now  for  suitable $0\leq f\in L^{q}(\Omega)$, one can consider again the solution $u$ to 
	\begin{equation*}
	\begin{cases}
	\displaystyle -\Delta u= f &  \text{in}\, \Omega, \\
	u=0 & \text{on}\ \partial \Omega,
	\end{cases}
	\end{equation*}
 such that  $u$ is $W^{1,q^{*}}_{0}(\Omega)$ but $u\not\in H^1_0(\Omega)$. 
 Let $h$ satisfying both \eqref{h1} and \eqref{h2} with, resp.,  $\gamma>0$ and $\theta<1$. Then $u$ satisfies 
	\begin{equation*}
	\begin{cases}
	\displaystyle -\Delta u= h(u) g &  \text{in}\, \Omega, \\
	u=0 & \text{on}\ \partial \Omega,
	\end{cases}
	\end{equation*}
where $g=fh(u)^{-1}$. 	We claim that  $g$ in $L^m (\Omega)$; indeed, recalling that $m<q$ and also using  H\"older's inequality, one has 
	$$
	\begin{array}{l}\dys \into  g^{m}= \int_{\{u<\underline{\omega}\}} f^{m}u^{\gamma m} + \int_{\{\underline{\omega}\leq u\leq \overline{\omega}\}}f^{m}h(u)^{-m} +\int_{\{u>\overline{\omega}\}}f^{m}u^{\theta m} \\  \\
	\dys \leq \underline{\omega}^{\gamma m }\into f^{m}  +  \max_{[\underline{\omega},\overline{\omega}]}( {h(s)}^{-m})\into f^{m} + C_{}\left(\into u^{\frac{\theta m q}{q-m}}\right)^{\frac{q-m}{q}} \leq C 
	\end{array}$$
	since, thanks to the  choices of both $q$ and $m$, we have   $\frac{\theta m q}{q-m}<\frac{Nq}{N-2q}= q^{**}$. We then have nonnegative  data in  $L^{m}(\Omega)$ \bk  for which solutions to our problem need  not to belong to $\huz$,  for $\theta<1$.   Observe that the behavior at zero (i.e. $\gamma$) plays no role and that $m$ needs to be  close to $1$ (to be compared with the extra assumptions required  in Theorem \ref{gamma<1}).  
\end{example}

As we saw, the model case (i. e. $h(s)=s^{-\gamma}$, $\gamma>1$) of problem \eqref{1} is covered by  Theorem \ref{teosharp}. The proof is based on Theorem \ref{cino} which is not known for a general nonlinearity $h$; though the proof in \cite{sz}  can not be directly generalized due to some technical homogeneity issues, one can conjecture that it still holds, at least if $h(s)s$ is nonincreasing. That is, in this case, one could have that \eqref{1} has a finite energy solution if and only if there exists $u_{0}\in\huz$ such that 
$$
\into f h(u_{0})u_{0} <\infty\,.
$$

\medskip 

We showed that,  for strongly singular  (at infinity) $h$'s then solutions need not to have finite energy in general. In the model case, one also has  the following natural  regularity criterion:
\begin{proposition}
Let $\gamma<1$, $f\in \luo$,  and $u$ solves  \eqref{1cin} then $u\in\huz$ if and only if $$\into fu^{1-\gamma}<\infty\,.$$ 
\end{proposition}
\begin{proof}
$(\Rightarrow)$ Apply Theorem \ref{bocat} and use $u$ as test in \eqref{1cin}. 

$(\Leftarrow)$  One considers the approximating solutions $u_{n}$ to 
$$\begin{cases}
-{\rm div} (A(x)\nabla  u_n)= \frac{f_n}{(u_{n}+\frac{1}{n})^{\gamma}} & \text{in}\;\ \Omega,\\
u_n=0 & \text{on}\;\ \partial\Omega,
\end{cases}$$
 and use $u_{n}$ as test, to obtain
$$
\alpha\into |\nabla u_{n}|^{2}\leq \into \frac{f_{n}}{(u_{n}+\frac{1}{n})^{\gamma}}u_{n}\leq \into f_{n} u_{n}^{1-\gamma} \leq \into f u^{1-\gamma}\,,
$$
where in the last step we also used that $u_{n}$ increases with respect to $n$, and we conclude by weak  lower semicontinuity. 
\end{proof}

One could conjecture that this kind of regularity principle, specializing in some sense the one given by Theorem \ref{cino}, extends for any $\gamma>0$ (or, more, for any $h$). That is  one expects that, given a nonnegative $f\in L^{1}(\Omega)$ and the solution $u$  to \eqref{1f}, then 
$$u\in\huz\ \text{ if and only if}\ \  \into f h(u)u<\infty\,.$$
Notice that solutions in  Example \ref{ex} satisfy (sharply) this criterion. In fact, if  $\gamma>1$ then $\into fu^{1-\gamma}<\infty$ if and only if $\gamma<3-\frac{2}{m}$ if and only if $u\in\huz$.

\medskip

 Consider now the threshold $\frac{1}{1-\gamma}$ given   in Theorem \ref{gamma<1L1}. As we said it can not be improved (at Lebesgue's scale) in view of  Example \ref{ex}. Although, the borderline case remains open. A straightforward modification in the proof of Theorem \ref{gamma<1L1},  using suitable generalized H\"older's inequalities,  shows that the result still holds true if $f$ belongs to the Lorentz space $L^{\frac{1}{1-\gamma},1}(\Omega)$. 
 The following  example contains an explicit instance of  this occurrence in  the case  of problem  \eqref{1f}. 
\bk 
\begin{example}
Let  $h$ satisfying \eqref{h1} with $\gamma<1$, $a>1$, and consider  \begin{equation}f(x)=\frac{1}{\delta(x)^{1-\gamma}(-\log (\delta(x)))^{a}}\,.\label{f}\end{equation} 

Without loss of generality we are assuming ${\rm diam} (\Omega)<e^{-1}$;  otherwise we define $f$ as in \eqref{f} near the boundary and we suitably truncate it  inside of $\Omega$. Reasoning as in the proof of Theorem \ref{gamma<1L1} one can show that $ u\geq v_{1}\geq C\delta $ where $u$ is the solution to \eqref{1f}.  Then, as $a>1$,  one has 
$$ \begin{array}{l}
\dys  \into  h(u)f\leq  \int_{\{u<\underline{\omega}\}} f u^{-\gamma} + \int_{\{u\geq \underline{\omega}\}}  h(u)f  \leq \int_{\{u<\underline{\omega}\}} f v_{1}^{-\gamma}+\sup_{[\underline\omega,\infty)} h(s)\int_{\{u\geq \underline{\omega}\}} f  \\ \\ \dys\leq C^{-\gamma} \into \frac{1}{\delta(-\log (\delta))^{a}}+C<\infty\,. \end{array}
$$
\end{example}

\bk
\medskip 

One last  interesting remark concerns the way the homogeneous boundary datum  is assumed.  As we observed, in general, one is not able to prove that the solution has a trace in the classical sense (at least if $\gamma>1$) and we get rid of this fact by introducing the relaxed boundary condition \eqref{bd}. Anyway, in view  of Lemma \ref{M1respectdelta},  we are again  in presence of  a  borderline case. Consider for simplicity a function $f$ as  datum;   in \cite{dr} the authors prove that, if the lower order term ($h(u)f$ in our case) belongs to $L^1(\Omega, \delta^\alpha)$, for some $0<\alpha<1$, then $u\in W^{1,q}_{0}(\Omega)$, for some $q>1$. Moreover, if   $h(u)f\in L^1 (\Omega, \delta|\log(\delta)|)$, then  $u\in W^{1,1}_{0}(\Omega)$. The solutions of Example \ref{ex} are continuous up to the boundary and  we are far from this pathological behavior; although, only observe that in that case one always has that $fu^{-\gamma}\in L^1(\Omega, \delta^\alpha)$, for $\alpha>\frac{\gamma}{1+\gamma}$.

\section{Appendix}
\setcounter{equation}{0}
\subsection{Proof of Theorem \ref{bocat}}

	Let $\phi\in \huz $ and consider  $\psi_n$  a sequence in $C^1_c(\Omega)$ that converges to  $\phi$ in $H^1_0(\Omega)$. For $\vare>0$ we take   $\varphi= \sqrt{\vare^2+|\psi_n-\psi_k|^2}-\vare$ as a test function in the distributional formulation of $u$. 
	Using H\"older's inequality one gets 
$$\begin{array}{l}
		\displaystyle\int_{\Omega} h(u)f(\sqrt{\vare^2+|\psi_n-\psi_k|^2}-\vare) = \int_{\Omega}A(x)\nabla u\cdot  \frac{\nabla(\psi_n-\psi_k)(\psi_n-\psi_k)}{\sqrt{\vare^2+|\psi_n-\psi_k|^2}} \nonumber \\  \dys \le \beta||u||_{H^1_0(\Omega)}\left(\int_{\Omega}\frac{|\nabla(\psi_n-\psi_k)|^2(\psi_n-\psi_k)^2}{\vare^2+|\psi_n-\psi_k|^2}\right)^{\frac{1}{2}}\le \beta||u||_{H^1_0(\Omega)}||\psi_n-\psi_k||_{H^1_0(\Omega)}.
		\label{stimateoboc1}\end{array}
$$
	Therefore, by Fatou's Lemma
	\begin{equation}
		\displaystyle \int_{\Omega} \left|h(u)f\psi_n - h(u)f\psi_k \right| \le \beta||u||_{H^1_0(\Omega)}||\psi_n-\psi_k||_{H^1_0(\Omega)}.
		\label{stimateoboc2}
	\end{equation}
	Estimate \eqref{stimateoboc2} implies that the sequence $\displaystyle h(u)f\psi_n$ is a Cauchy sequence in $L^1(\Omega)$,  
	and that converges to its a.e.  limit  $\displaystyle h(u)f\phi$. We  then take $\psi_{n}$ as test in \eqref{distri} and we can pass to the limit  obtaining that  $u$ satisfies \eqref{H10test}. 
	
	 To prove uniqueness we consider  two solutions $u_1$ and  $u_2$  of \eqref{1f}  in $\huz$.    We are allowed to choose  $(u_1-u_2)^-$ as a test function in \eqref{H10test} for  both  $u_1, u_2$. Taking the difference we get 
	$$ \displaystyle \int_{\Omega}A(x)\nabla (u_1-u_2)\cdot\nabla (u_1-v_2)^- \\\\ \dys = \int_{\Omega} \left(h(u_1)f-h(u_2)f\right)(u_1-u_2)^- \ge 0$$
	that implies, using the ellipticity of $A$, that   $u_1\geq u_{2} $. The opposite inequality is obtained by interchanging the roles of $u_{1}$ and $u_{2}$.  \hfill $\square$

\begin{remark}
	Note that the proof of Theorem  \ref{bocat}  is essentially based on the ellipticity of the operator and so it can be easily extended to fairly general classes of problems as, for instance, the ones involving Leray-Lions type nonlinear operators in general domains. We also remark that, one do not need the knowledge of the behavior of $h$ neither at zero nor at the infinity (we only use that $h(u)$ is locally bounded).   
\end{remark}

\subsection*{Acknowledgements} We would like to warmly thank both Luigi Orsina and Augusto C. Ponce for fruitful discussions and, in particular, the ones concerning, resp.,   Example \ref{luigi} and  Lemma \ref{M1respectdelta}.


\end{document}